\documentclass[11pt,leqno]{article}
\usepackage{graphicx, amsfonts, amsthm, amsxtra, amssymb, verbatim, makeidx}
\usepackage{subeqnarray, relsize}
\usepackage[mathscr]{euscript}
\usepackage{hyperref}
\usepackage[ruled,vlined]{algorithm2e}
\makeindex
\makeglossary
\begin{document}
\newtheorem{theo}{Theorem}
\newtheorem{exam}{Example}
\newtheorem{coro}{Corollary}
\newtheorem{defi}{Definition}
\newtheorem{prob}{Problem}
\newtheorem{lemm}{Lemma}
\newtheorem{prop}{Proposition}
\newtheorem{rem}{Remark}
\newtheorem{conj}{Conjecture}
\newtheorem{calc}{}

%\include{../Biblio/notations}
%\include{notations}

%------Chapter on algebraic cycles-----------------------
\def\gru{\mu} %the group of d-th roots of unity 
\def\pg{{ \sf S}}               %permutation group
\def\TS{\mathlarger{\bf T}}                %Tangent space
\def\NB{{\mathlarger{\bf N}}}
\def\group{{\sf G}}
\def\NLL{{\rm NL}}   %Noether-Lefschetz locus

\def\plc{{ Z_\infty}}    %a section of X by a projective space of dimension n/2+1
\def\pola{{u}}      %polarization
\newcommand\licy[1]{{\mathbb P}^{#1}} %linear projective cycle
\newcommand\aoc[1]{Z^{#1}}     % Aoki-Shioda cycle
\def\HL{{\rm Ho}}     %Hodge locus
\def\NLL{{\rm NL}}   %Noether-Lefschetz locus

%----------------General Math Notations-------------------------------------
\def\Z{\mathbb{Z}}                   %Integer  numbers
\def\Q{\mathbb{Q}}                   %Rational  numbers
\def\C{\mathbb{C}}                   %Complex numbers
\def\N{\mathbb{N}}                   %natural numbers
\def\uhp{{\mathbb H}}                %upper half plane
\def\A{\mathbb{A}}                   %affine space C^n
\def\dR{{\rm dR}}                    %The subindex dR standing for de Rham cohomology.
\def\F{{\cal F}}                     %A foliation
\def\Sp{{\rm Sp}}                    %Symplectic group
\def\Gm{\mathbb{G}_m}                 %The multiplicative group
\def\Ga{\mathbb{G}_a}                 %The additive  group
\def\Tr{{\rm Tr}}                      %Trace map in Algebraic de Rham cohomology
\def\tr{{{\mathsf t}{\mathsf r}}}                 %Transposition of matrices
\def\spec{{\rm Spec}}            %The spectrume
\def\ker{{\rm ker}}              %kernel
\def\GL{{\rm GL}}                %The liner group
\def\ker{{\rm ker}}              %kernel
\def\coker{{\rm coker}}          %cokernel
\def\im{{\rm Im}}               %Image
\def\coim{{\rm Coim}}            %coimage
\def\p{{\sf  p}}
\def\U{{\cal U}}   %covering

\def\weig{{\nu}}
\def\r{{ r}}                       %dimension of the moduli space of hypersurfaces.
%----------------Gauss-Manin Connection in disguise---------------------
\def\k{{\sf k}}                     %Arbitrary field
\def\ring{{\sf R}}                   %A ring
\def\X{{\sf X}}                      %families of varieties
\def\Ua{{   L}}                      %families of affine varieties
\def\T{{\sf T}}                      %Moduli of enhanced  varieties
\def\asone{{\sf A}}                  %affine space A^{n+1} defined over a ring

\def\Ts{{\sf S}}
\def\cmv{{\sf M}}                    %Classical moduli of varieties
\def\BG{{\sf G}}                       %Borel Algebraic Group
\def\podu{{\sf pd}}                   %Poincare dual
\def\ped{{\sf U}}                    %Period Domain
\def\per{{\bf  P}}                   %period matrix or fundamental system of GMC
\def\gm{{  A}}                    %Gauss-Manin connection
\def\gma{{\sf  B}}                   %Gauss-Manin connection
\def\ben{{\sf b}}                    %Betti number

\def\Rav{{\mathfrak M }}                     % Space of modular vector fields
\def\Ram{{\mathfrak C}}                     % Space of constant vector fields
\def\Rap{{\mathfrak G}}                     % Space of vector fields arising from group  action. 

\def\mov{{\sf  m}}                    %Fixed dimension of the cohomology
\def\Yuk{{\sf C}}                     %Yukawa and generalizations 
\def\Ra{{\sf R}}                      %Ramanujan type vecto field
\def\hn{{ h}}                         %Hodge numbers
\def\cpe{{\sf C}}                     %Constant periods
\def\g{{\sf g}}                       %An element of the Borel group
\def\t{{\sf t}}                       %An element \T
\def\pedo{{\sf  \Pi}}                  %Period domain before discrete group action

\def\Der{{\rm Der}}                   %Derivations
\def\MMF{{\sf MF}}                    %Moduli of modular foliations
\def\codim{{\rm codim}}                %codimension
\def\dim{{\rm    dim}}                %dimension
\def\Lie{{\rm Lie}}                   %Lie algebra of a group.
\def\gg{{\mathfrak g}}                %elements of a Lie group

\def\u{{\sf u}}                       %An element of the period domain

\def\imh{{  \Psi}}                 %intersection matrix in homology
\def\imc{{  \Phi }}                  %intersection matrix in cohomology
\def\stab{{\rm Stab }}               %stablizer 
\def\Vec{{\rm Vec}}                 %space of vector fields
\def\prim{{\rm  0}}                  %primitive cohomology

\def\Fg{{\sf F}}     %Genus g topological partition function
\def\hol{{\rm hol}}  %holomorphic
\def\non{{\rm non}}  %non-holomorphic
\def\alg{{\rm alg}}  %algebraic
\def\tra{{\rm tra}}  %transcendental

\def\bcov{{\rm \O_\T}}       %The ring of modular-type functions

\def\leaves{{\cal L}}        %space of leaves  

\def\cat{{\cal A}}              %category
\def\im{{\rm Im}}               %Image

%%%%%%%%%%%%%%%%%%%%%%%%%%%%%%%%%%%%%%%%%%%%%%%%%%%%%%%%%%%%%%%%%
\def\pn{{\sf p}}              %Periods of Hodge cycles
\def\Pic{{\rm Pic}}           %Picard group 
\def\free{{\rm free}}         %Free part
\def \NS{{\rm NS}}    %Neron-Severi group 
\def\tor{{\rm tor}}
\def\codmod{{\xi}}    %Codimension of the Hodge loci

%---------------OLD Notation-------------------
\def\GM{{\rm GM}}

\def\perr{{\sf q}}        %period matrix.....
\def\perdo{{\cal K}}   %period domain
\def\sfl{{\mathrm F}} %Space of filtrations
\def\sp{{\mathbb S}}  %Sphere

\newcommand\diff[1]{\frac{d #1}{dz}} %Differential operator
\def\End{{\rm End}}              %Endomorphism group

\def\sing{{\rm Sing}}            %The set of singularities
\def\cha{{\rm char}}             %Charracteristic
\def\Gal{{\rm Gal}}              %The Galois group
\def\jacob{{\rm jacob}}          %the Jacobian ideal
\def\tjurina{{\rm tjurina}}      %the tjurina ideal
\newcommand\Pn[1]{\mathbb{P}^{#1}}   %Projective space of dimension #1
\def\P{\mathbb{P}}
\def\Ff{\mathbb{F}}                  %Finite field

\def\O{{\cal O}}                     %ring of integers of a number field

\def\ring{{\mathsf R}}                         %A ring
\def\R{\mathbb{R}}                   %real numbers

\newcommand\ep[1]{e^{\frac{2\pi i}{#1}}}% unipotent numbers
\newcommand\HH[2]{H^{#2}(#1)}        %Hodge structures
\def\Mat{{\rm Mat}}              %Matrices
\newcommand{\mat}[4]{
     \begin{pmatrix}
            #1 & #2 \\
            #3 & #4
       \end{pmatrix}
    }                                %two by two matrices
\newcommand{\matt}[2]{
     \begin{pmatrix}                 % one by two matrix
            #1   \\
            #2
       \end{pmatrix}
    }
\def\cl{{\rm cl}}                %Chern class

\def\hc{{\mathsf H}}                 %The set of Hodge cycles.
\def\Hb{{\cal H}}                    %Hodge bundle
\def\pese{{\sf P}}                  %Period set

\def\PP{\tilde{\cal P}}              %the period domain/ discrete group
\def\K{{\mathbb K}}                  %Field representing R or C

\def\M{{\cal M}}
\def\RR{{\cal R}}
\newcommand\Hi[1]{\mathbb{P}^{#1}_\infty}%the hyperplane at infinity
\def\pt{\mathbb{C}[t]}               %Polynomials in t
\def\gr{{\rm Gr}}                %graded pieces
\def\Im{{\rm Im}}                %imaginary
\def\Re{{\rm Re}}                %Real
\def\depth{{\rm depth}}
\newcommand\SL[2]{{\rm SL}(#1, #2)}    %SL(2,Z)
\newcommand\PSL[2]{{\rm PSL}(#1, #2)}  %PSL(2,Z)
\def\Resi{{\rm Resi}}              %Residue

\def\L{{\cal L}}                     %The moduli of polarized lattices in a
                                     %fixed vector spaces.
\def\Aut{{\rm Aut}}              %Automorphism group of a vectorspace
\def\any{R}                          %Any subring of the field of complex
                                     %numbers.
\newcommand\ovl[1]{\overline{#1}}    %Conjugation of #1.

\newcommand\mf[2]{{M}^{#1}_{#2}}     %New modular functions
\newcommand\mfn[2]{{\tilde M}^{#1}_{#2}}     %New modular functions

\newcommand\bn[2]{\binom{#1}{#2}}    %Binomial
\def\ja{{\rm j}}                 %j of a two by two matrix
\def\Sc{\mathsf{S}}                  %Simple cycles
\newcommand\es[1]{g_{#1}}            %Eisenstein series
\newcommand\V{{\mathsf V}}           %Milnor vector space
\newcommand\WW{{\mathsf W}}          %Similar to Milnor vector space
\newcommand\Ss{{\cal O}}             %Structural sheaf
\def\rank{{\rm rank}}                %rank of a module
\def\Dif{{\cal D}}                   %Differentials
\def\gcd{{\rm gcd}}                  %greatest common divisor
\def\zedi{{\rm ZD}}                  %zero divisors of a module
\def\BM{{\mathsf H}}                 %Brieskorn module
\def\plf{{\sf pl}}                             %Picard-Lefschetz formula
\def\sgn{{\rm sgn}}                      %sign
\def\diag{{\rm diag}}                   %diagonal matrix
\def\hodge{{\rm Hodge}}
\def\HF{{ F}}                                %The hodge filtration of the brieskon module
\def\WF{{ W}}                               %The weight filtration of the brieskon module
\def\HV{{\sf HV}}                                %humbert variety
\def\pol{{\rm pole}}                               %pole divisor
\def\bafi{{\sf r}}
\def\id{{\rm id}}                               %identity
\def\gms{{\sf M}}                           %Gauss-Manin system
\def\Iso{{\rm Iso}}                           %Gauss-Manin system

\def\hl{{\rm L}}    %holomorphic limit
\def\imF{{\rm F}}
\def\imG{{\rm G}}

\def\disc{{\rm disc}}

\begin{center}
{\LARGE\bf  Integral Hodge conjecture for Fermat varieties
%\footnote{ 
%Math. classification: 14N35, %Gromov-Witten invariants, quantum cohomology, Gopakumar-Vafa invariants, Donaldson-Thomas invariants 
%14J15, 32G20%Period matrices, variation of Hodge structure; degeneration 
%\\
%Keywords: Gauss-Manin connection, Holomorphic foliations, Hodge filtration, Griffiths transversality. 
%}
}
\\
\vspace{.25in} {\large {\sc Enzo Aljovin, Hossein Movasati, Roberto  Villaflor Loyola}}\footnote{
Instituto de Matem\'atica Pura e Aplicada, IMPA, Estrada Dona Castorina, 110, 22460-320, Rio de Janeiro, RJ, Brazil,
{\tt %www.impa.br/$\sim$ hossein, 
aljovin@impa.br, hossein@impa.br, rvilla@impa.br}}
\end{center}

\begin{abstract}
We describe an algorithm which verifies whether  linear algebraic cycles of the Fermat variety  
generate the lattice of Hodge cycles. A computer implementation of this
confirms the integral Hodge conjecture for quartic and quintic Fermat fourfolds. Our algorithm is based on computation of the list of elementary
divisors of both the lattice of linear algebraic cycles, and the lattice of Hodge cycles written in terms of  vanishing cycles, 
and observing that these two lists are the same. 
%This suggests that such elementary divisors are  invariants of the lattice of Hodge cycles itself and do not depend on a given basis of the lattice. 
\end{abstract}

\section{Introduction}

The integral Hodge conjecture was formulated by Hodge in 1941 \cite{Hodge1952}, and it states that every integral Hodge cycle is algebraic.
In 1962, Atiyah and Hirzebruch \cite{ati62} showed that not every torsion cycle 
is algebraic, providing the first counterexamples to integral Hodge conjecture and restating it over rational numbers.  For a survey on rational Hodge conjecture, which  is still an open problem, see \cite{Lewis}. 
Even for many varieties for which  the Hodge conjecture is known to be true, the integral Hodge conjecture remains open. Koll\'ar in \cite{kol92} shows that  integral Hodge conjecture fails for very general threefold in $\mathbb{P}^4$ of high enough degree. Beside surfaces, integral Hodge conjecture is known to be true for cubic fourfolds, see \cite[Theorem 2.11]{voisin2013}. For improvements in Atiyah-Hirzebruch approach  see \cite{tot97} and \cite{SV2005}, see also \cite{CtV12} and \cite{Totaro2013} for counterexamples in the direction of Hassett-Tschinkel method \cite[Remarque 5.10]{CtV12}, which is an improvement of Koll\'ar's construction. In the present article we prove the integral Hodge conjecture for a class of Fermat varieties.

Let $n\in \N$ be an even number and
\begin{equation}
\label{27nov2015}
X=X^d_n\subset \Pn {n+1}:   \ \  x_0^{d}+x_1^{d}+\cdots+x_{n+1}^d=0, 
\end{equation}
be the Fermat variety of dimension $n$ and degree $d$.
It has the following well-known algebraic cycles of dimension $\frac{n}{2}$ which we call linear algebraic cycles:
\begin{equation}
\label{ref2}
 \licy{\frac{n}{2}}_{a,b}:  
\left\{
 \begin{array}{l}
 x_{b_0}-\zeta_{2d}^{1+2a_1}x_{b_1}=0,\\
 x_{b_2}-\zeta_{2d}^{1+2a_3} x_{b_3}=0,\\
 x_{b_4}-\zeta_{2d}^{1+2a_5} x_{b_5}=0,\\
 \vdots \\
 x_{b_n}-\zeta_{2d}^{1+2a_{n+1}} x_{b_{n+1}}=0,
 \end{array}
 \right. 
\end{equation}
where $\zeta_{2d}$ is a $2d$-primitive root of 
unity, $b$ is a permutation of $\{0,1,2,\ldots,n+1\}$ and  $0\leq a_i\leq d-1$ are integers.
In order to get distinct cycles we may further assume that $b_0=0$ and for $i$ an even number, $b_i$ is the smallest number in
$\{0,1,\ldots,n+1\}\backslash\{b_0,b_1,b_2,\ldots, b_{i-1}\}$. 
It is easy
to check that the number of such cycles is $$
N:=1\cdot 3\cdots (n-1)\cdot (n+1)\cdot  d^{\frac{n}{2}+1},
$$ 
(for $d=3, n=2$ this is the famous $27$ lines in a smooth cubic surface). In \S\ref{camacho-fgv} we introduce a lattice 
$V^d_n\subset H_n(X,\Z)$ such that if 
\begin{equation} 
\label{1nov2017}
d  \hbox{ is a prime number,   } \hbox{ or } d=4,    \hbox{ or }  \gcd(d, (n+1)!)=1
\end{equation}
then  by Poincar\'e duality $V^d_n\cong H^n(X,\Z)\cap H^{\frac{n}{2},\frac{n}{2}}$. We prove the following theorem:
\begin{theo}
\label{11oct2017}
For $(n,d)$ in Table \ref{elementary-divs} we have: 
\begin{enumerate}
\item
The topological classes of $\licy{\frac{n}{2}}_{a,b}$ generate the lattice  
$V^d_n\subset H_n(X,\Z)$ and so for those $(n,d)$ with the property \eqref{1nov2017} 
the integral Hodge conjecture is valid (this includes the cases $(n,d)=(4,4),(4,5)$). 
\item
The elementary  divisors and the discriminant  of such a lattice are listed in the same table.
\end{enumerate}
\end{theo}

\begin{table}[ht!]
\centering
\begin{tabular}{lll}
\hline
\multicolumn{1}{ | l | }{$(n, d)$}			& \multicolumn{1}{l | }{Elementary divisors/discriminant}	&	 \multicolumn{1}{l | }{Rank}  \\			\hline
\multicolumn{1}{ | l | }{$(2,3)$}			& \multicolumn{1}{l | }{$+1^{7}$} & \multicolumn{1}{l | }{7}  \\ \hline
\multicolumn{1}{ | l | }{$(2,4)$}			& \multicolumn{1}{l | }{$-1^{18} \cdot 8^{2}$} & \multicolumn{1}{l | }{20} \\ \hline
\multicolumn{1}{ | l | }{$(2,5)$}			& \multicolumn{1}{l | }{$+1^{26} \cdot 5^{10} \cdot 25^{1}$} & \multicolumn{1}{l | }{37} \\ \hline
\multicolumn{1}{ | l | }{$(2,6)$}			& \multicolumn{1}{l | }{$+1^{37} \cdot 3^{13} \cdot 12^{8} \cdot 36^{3} \cdot 108^{1}$} &\multicolumn{1}{l | }{62}  \\ \hline
\multicolumn{1}{ | l | }{$(2,7)$}			& \multicolumn{1}{l | }{$-1^{48} \cdot 7^{38} \cdot 49^{5}$} &\multicolumn{1}{l | }{91}  \\ \hline
\multicolumn{1}{ | l | }{$(2,8)$}			& \multicolumn{1}{l | }{$-1^{54} \cdot 2^{12} \cdot 8^{48} \cdot 16^{2} \cdot 32^{8} \cdot 64^{4}$} &\multicolumn{1}{l | }{128}  \\ \hline
\multicolumn{1}{ | l | }{$(2,9)$}			& \multicolumn{1}{l | }{$-1^{72} \cdot 3^{7} \cdot 9^{72} \cdot 27^{7} \cdot 81^{11}$} &\multicolumn{1}{l | }{169}  \\ \hline
\multicolumn{1}{ | l | }{$(2,10)$}			& \multicolumn{1}{l | }{$-1^{80} \cdot 5^{18} \cdot 10^{96} \cdot 20^{13} \cdot 100^{10} \cdot 500^{1}$} & \multicolumn{1}{l | }{218} \\ \hline
\multicolumn{1}{ | l | }{$(2,11)$}			& \multicolumn{1}{l | }{$+1^{96} \cdot 11^{158} \cdot 121^{17}$} & \multicolumn{1}{l | }{271} \\ \hline
\multicolumn{1}{ | l | }{$(2,12)$}			& \multicolumn{1}{l | }{$+1^{109} \cdot 3^{5} \cdot 12^{192} \cdot 24^{2} \cdot 48^{2} \cdot 144^{21} \cdot 432^{1}$}& \multicolumn{1}{l | }{332} \\ \hline
\multicolumn{1}{ | l | }{$(2,13)$}			& \multicolumn{1}{l | }{$+1^{120} \cdot 13^{254} \cdot 169^{23}$} &\multicolumn{1}{l | }{397}  \\ \hline
\multicolumn{1}{ | l | }{$(2,14)$}			& \multicolumn{1}{l | }{$-1^{126} \cdot 7^{20} \cdot 14^{288} \cdot 28^{15} \cdot 196^{20} \cdot 1372^{1}$} & \multicolumn{1}{l | }{470} \\ \hline
\multicolumn{1}{ | l | }{$(4,3)$}			& \multicolumn{1}{l | }{$+1^{19} \cdot 3^{1} \cdot 9^{1}$} & \multicolumn{1}{l | }{21} \\ \hline
\multicolumn{1}{ | l | }{$(4,4)$}			& \multicolumn{1}{l | }{$+1^{100} \cdot 2^{2} \cdot 4^{4} \cdot 8^{30} \cdot 16^{4} \cdot 32^{2}$} &\multicolumn{1}{l | }{142}  \\ \hline
\multicolumn{1}{ | l | }{$(4,5)$}			& \multicolumn{1}{l | }{$+1^{166} \cdot 5^{174} \cdot 25^{54} \cdot 125^{7}$} & \multicolumn{1}{l | }{401} \\ \hline
\multicolumn{1}{ | l | }{$(4,6)$}			& \multicolumn{1}{l | }{$-1^{310} \cdot 3^{272} \cdot 6^{40} \cdot 12^{163} \cdot 36^{116} \cdot 108^{61} \cdot 216^{24} \cdot 648^{16}$} & \multicolumn{1}{l | }{1002} \\ \hline
\end{tabular}
\caption{Elementary divisors of the lattice of linear Hodge cycles for $X^d_n$}
\label{elementary-divs}
\end{table}

The cases in the first part of Theorem \ref{11oct2017} for Fermat surfaces ($n=2$) are  particular instances of 
a general theorem proved in \cite{Schuett2010} and \cite{Degtyarev2015} for $d=4$ or ${\rm gcd}(d,6)=1$. 
The integral Hodge conjecture for cubic 
fourfolds is proved in \cite{voisin2013}, see Theorem 2.11 and the comments thereafter. Nevertheless, note that Theorem \ref{11oct2017} in this case does not follow from this. 
The rational Hodge
conjecture in all the cases covered in our theorem is well-known, see \cite{sh79}. 
The limitation for $(n,d)$ in our main theorem is due to the fact that its proof
uses the computation of Smith normal form of integer matrices, for further discussion 
on this see \S\ref{16oct2017swap}.

After the first draft of this article was written,  A. Degtyarev informed us 
about his joint paper \cite{Degtyarev2016} with I. Shimada in which the authors give a criterion for primitivity of the sublattice of  $H_n(X,\Z)$ generated
by linear algebraic cycles. They also use a computer to check the primitivity in many cases of $(n,d)$ in our table, see \cite[Section 5]{Degtyarev2016}. 
This together with Shioda's result implies
the first statement in Theorem \ref{11oct2017}. The verification of this mathematical statement by computer, using two different methods and 
different people, bring more trust on such a statement even without a theoretical proof. 
Their method is based on computation of the intersection number of a linear algebraic cycle 
with a vanishing cycle (\cite[Theorem 2.2]{Degtyarev2016}), 
whereas our method is based on the computation of both intersection matrices of vanishing cycles, see \eqref{27.1.2017},  
and linear algebraic cycles, see \eqref{13jan16-2}.   The advantage of our method is the computation of  elementary divisors and the discriminant of the lattice of 
Hodge cycles. %  is the main contribution of the present article to the literature. 

In Table \ref{elementary-divs}, $a^b\cdot c^d\cdots$ means $a$, $b$ times, followed by $c$, $d$ times etc. Interpreting this as usual multiplication
we get the discriminant of the lattice of Hodge cycles.
The discriminant is an invariant of a lattice and it is easy to see that the list of elementary divisors are also, see \cite[Theorem 2.4.12]{Cohen1993}. 
%see Proposition \ref{3nov2017}.
We have computed such elementary divisors from two completely different bases of $V^d_n$. The first one is computed using the intersection matrix of linear algebraic cycles and the other from the 
intersection matrix of Hodge cycles written in terms of vanishing cycles. 
%This makes us to believe that for the lattice $V^d_n$ the whole data of elementray divisors are invariants of the  lattice $V^d_n$. 
We can give an interpretation of the power of $1$ in the list of elementary divisors, see \S\ref{1.11.17}. For instance,  the lattice 
$V^6_4$ modulo $3$ (resp. $2$) is of discriminant zero and its largest sub lattice of non-zero discriminant has rank $310$ (resp. $310+272$).
In all the cases in Table \ref{elementary-divs}, for primes $p$ such that $p$ does not divide $d$, the lattice $V^d_n$ modulo $p$ has non-zero discriminant.  
Finally,  it is worth mentioning that I. Shimada  in \cite[Theorem 1.5 and 1.6]{Shimada2001}  for the case $(n,d)=(4,3)$ and in characteristic $2$ proves that the lattice 
spanned by linear algebraic cycles is isomorphic to the laminated lattice of rank $22$ . He also computes the discriminant of such a lattice in
the case $(n,d)=(6,3)$. For an explicit basis of $V^4_2\cong H^2(X^4_2,\Z)\cap H^{1,1}$ using linear algebraic cycles  and the corresponding 
intersection matrix see \cite[Table 3.1]{ShimadaShioda2017}.

%I. Shimada kindly pointed out
%that the Gram matrix of the lattice of Hodge cycles in the case $n=2, d=4$ is presented  in \cite{ShimadaShioda2017}
%and for $(n,d)=(4,3)$ and characteristic $2$  the primitive part of the . 
%The verification
%of the integral Hodge conjecture for quartic and quintic  fourfolds by linear algebraic cycles, 
%and the computation of corresponding elementary divisors in all cases in Table \ref{elementary-divs}, 
%are the main contribution of this article to the literature. 

\section{Preliminaries on lattices}
\subsection{Smith normal form:}
Let $\tt A$ be a $\mu\times \check\mu$ matrix with integer coefficients. 
There exist  $\mu\times \mu$  and $\check\mu\times\check\mu$ integer-valued matrices $\tt U$ and  $\tt T$, respectively,  
such that 
$\det(\tt U)=\pm 1$, $\det(\tt T)=\pm 1$ and 
\begin{equation}
\label{13oct2017}
\tt U\cdot A \cdot T=\tt S
\end{equation}
where $\tt S$ is the diagonal matrix $\diag(a_1,a_2,\ldots,a_m,0,\cdots,0)$. 
%$$
%{\tt S}:=
%\begin{pmatrix}
%a_1 & 0 & 0 &  \cdots & & & \\
%0 & a_2 & 0 & \cdots  & &  &\\
%0 & 0 & a_3 &    & & &\\
%\vdots & \vdots&   &  \ddots &    &  &\\
%       &       &        &         & a_{m}&  & \\
%       &       &        &         &    &  0 &  \\
%       
%      &       &        &         &    &   & \ddots \\
%\end{pmatrix}.
%$$
Here, $m$ is the rank of the matrix $\tt A$ and  the natural numbers $a_i\in\N$ satisfy 
$a_i \mid a_{i+1}$ for all $i < m$ and are called the elementary divisors.  The matrix $\tt S$ is called the Smith normal form of $\tt A$. 
The equality \eqref{13oct2017} is also called the Smith decomposition of ${\tt A}$. 
%The command {\tt SmithDecomposition} in {\tt Mathematica} is designed for this computation. 

\subsection{Elementary divisors of a Lattice:}
A lattice $V$ is a free finitely generated $\Z$-module equipped with a symmetric bilinear form $V\times V\to \Z$ which we also call the intersection bilinear map. 
We usually  fix a basis $v_1,v_2,\ldots,v_\mu$ of $V$. Its discriminant (resp. elementary divisors) is the determinant (resp. elementary divisors) of the intersection 
matrix  $[v_i\cdot v_j]$. The discriminant does not depend 
on the choice of the basis $v_i$, neither do the elementary divisors. 
Let $V$ be a lattice, possibly of discriminant zero. Let also  $V^\perp$ be the 
sublattice of $V$ containing 
all elements $v\in V$ such that $v\cdot V=0$.  The quotient $\Z$-module  $V/V^{\perp}$  is free and it inherits a bilinear form from $V$, and so it is a lattice.
Therefore, by definition $\rank(V/V^{\perp})\leq \rank(V)$ and $\disc(V/V^{\perp})\not=0$. 
Let $V$ be a lattice and its intersection matrix ${\tt A}=[v_i\cdot v_j]$ in 
a basis $v_i$ of $V$. 
In order to find  a basis, discriminant, and elementary divisors of $V/V^{\perp}$  we use the Smith normal form of ${\tt A}$.
We make the change
of basis $[w_i]_{\mu\times 1}:={\tt U}[v_i]_{\mu\times 1}$ and the intersection matrix is now 
$[w_i\cdot w_j]={\tt U}\cdot {\tt A}\cdot {\tt U}^{\tr}=\tt S\cdot \tt T^{-1}\cdot {\tt U}^{\tr}$. Therefore, using Smith normal form of matrices
we can always find a basis $w_i$ of $V$ with
\begin{equation}
\label{1.n.2017}
 [w_i\cdot w_j]=\tt S\cdot \tt R,
\end{equation}
where $\tt R$ is  $\mu\times \mu$  integer-valued matrix with $\det(\tt R)=\pm 1$ and $\tt S$ is a diagonal matrix as in \eqref{13oct2017}. 
Since $[w_i\cdot w_j]$ is a symmetric matrix, it follows that  $\tt R$ is a block diagonal matrix with two blocks of size $m\times m$ and $(\mu-m)\times(\mu-m)$ in the diagonal. 
Both blocks have determinant $\pm 1$ because $\det({\tt R})=\pm 1$. 
A basis of $V^\perp$  (resp. $V/V^{\perp}$) is given by $w_i,\ \ i=m+1,\cdots,\mu$ (resp. $i=1,2,\cdots, m$). It follows that the 
elementary divisors of $V/V^{\perp}$ in this basis are  $a_1,a_2,\cdots,a_m$. 

\subsection{Lattices mod primes}
\label{1.11.17}
For a lattice $V$ let $V_p:=V/pV$  which is a  $\Ff_p$-vector
space with the induced  $\Ff_p$-bilinear map $V_p\times V_p\to\Ff_p$. In a similar way as before we define  $V_p^\perp$. 
\begin{prop}
\label{2nov2017}
 Let $a_1,a_2,\ldots,a_m$ be the elementary divisors of the lattice  $V$. For a prime number such that $p\mid a_{s+1}$ but $p\not | a_j$ for all 
 $j\leq s$, the number $s$ is the rank of $V_p/V_p^\perp$. 
\end{prop}
\begin{proof}
This follows from the intersection matrix of $V$ with respect to the basis $w_i$ of $V$ in  \eqref{1.n.2017}.  
\end{proof}

\section{Fermat variety}
\label{camacho-fgv}
Let $X\subseteq \P^{n+1}$ be the Fermat variety of degree $d$ and even dimension $n$. In $H_n(X,\Z)$ we have the algebraic cycle $[\plc]$, where $\plc$ is the intersection of a linear projective space $\P^{\frac{n}{2}+1}$ with $X$. Every algebraic subvariety $Z\subseteq X$ of codimension $\frac{n}{2}$ induces an algebraic cycle $[Z]\in H_n(X,\Z)$. By abuse of notation, we will also denote the algebraic cycle by $Z$ instead of $[Z]$, letting it be understood in the context which one we are referring to. 
A cycle $\delta\in H_n(X,\Z)$ is called primitive if its intersection with $\plc$ is zero. For any subset $V\subset H_n(X,\Z)$
we denote by $V_\prim\subset V$ the set of primitive elements of $V$. 
Let $v_1,v_2,\cdots, v_\mu$ be any collection of primitive 
cycles (possibly linearly dependent in $H_n(X,\Z)$) and let $V$ be the free $\Z$-module generated by $v_i$'s. It has a natural lattice structure induced by the intersection of $v_i$'s in $H_n(X,\Z)$. From the Hodge index theorem it follows that
$V^\perp$ maps to zero in $H_n(X,\Z)$ and $V/V^\perp$ is a sublattice of $H_n(X,\Z)$. We will use this fact twice, for linear algebraic cycles and for 
Hodge cycles written in terms of vanishing cycles. 

\subsection{The lattice of linear algebraic cycles}
Let us index arbitrarily the linear subspaces defined in \eqref{ref2} as $P_{i},\ \ i=1,2,\ldots,N$.  A 
linear algebraic cycle $P_i$ intersects $\plc$ transversally in a point, and so  
$P_i\cdot\plc=1$. Moreover, 
one can use the adjunction formula and show 
\begin{equation}
 \label{13jan16-2}
P_i\cdot P_j=
\frac{1-(-d+1)^{m+1} }{d}
\end{equation}
where $\dim(P_i\cap P_j)=m$,  see \cite[Section 17.6]{ho13}.
The primitive part of the sublattice of $H_n(X,\Z)$ generated by $P_i$,
is generated by 
\begin{equation}
\label{15sept2017}
\check{P}_i:=P_i-P_1,\ \ i=2,\ldots,N. 
\end{equation}
The linear subspace $\P^{\frac{n}{2}+1}$ can be chosen as $\{x_0-\zeta_{2d}x_1=\cdots=x_{n}-\zeta_{2d}x_{n+1}=0\}$, so that the polarization cycle $\plc$ splits as a sum of $d$ linear algebraic cycles $P_i$ which intersect each other in a 
linear $\P ^{\frac{n}{2}-1}\subset \P ^{n+1}$. Therefore, the primitive cycle $dP_i-\plc$ is a $\Z$-linear 
combination of \eqref{15sept2017}. For the proof of Theorem \ref{11oct2017} we first write the Smith normal form of the $(N-1)\times (N-1)$ 
matrix
\begin{equation}
 \label{11.10.2017}
 {\tt A}_1:=[\check{P}_i\cdot  \check{P}_j].
\end{equation}

\subsection{The lattice of primitive Hodge cycles} \label{primitive_Hodge_cycles}
For all missing proofs and references in this paragraph see \cite[Chapter 15,16]{ho13}. 
Let $\Ua$ be the affine chart of the Fermat variety $X$ given by $x_0\not=0$. A topological cycle $\delta\in H_n(X,\Z)$ is primitive
if and only if it is supported in $\Ua$.  We also know that the homology group $H_n(\Ua,\Z)$ is free of rank 
$\mu:=(d-1)^{n+1}$ and we can choose a basis 
$$
\delta_{\beta},\ \beta\in I:=\{0,1,2,\ldots,d-2\}^{n+1}
$$ of $H_n(\Ua,\Z)$ by homological classes homeomorphic
to $n$-dimensional spheres, see \cite[\S 15.2]{ho13}. 
These are Lefschetz' famous vanishing cycles, which are obtained by the shifts of Pham polyhedron by deck translations, see \cite{Pham1965}. On the other hand, we can also choose a basis of the image of the restriction map $ F^{\frac{n}{2}+1}H^n_\dR(X)\to H^n_\dR(\Ua)$ (where $F^{\frac{n}{2}+1}$ is the 
$(\frac{n}{2}+1)$-th piece of the Hodge filtration of $H^n_\dR(X)$). In fact, take 
\begin{equation}
\label{dobareparishan-04/2016}
\omega_{\beta}:=x_1^{\beta_1}x_2^{\beta_2}\cdots x_{n+1}^{\beta_{n+1}}  \left(\sum_{i=1}^{n+1}
(-1)^{i-1} x_i dx_1\wedge dx_2\wedge\cdots\wedge \widehat {dx_i}\wedge \cdots \wedge dx_{n+1}\right), \beta\in I_1, \ 
\end{equation}
where
\begin{equation}
\label{indices_I1}
I_1:=\left\{ (\beta_1,\cdots,\beta_{n+1})\in\Z^{n+1} \Bigg|   0\leq \beta_i\leq d-2,\ \ 
\sum_{i=1}^{n+1}\frac{\beta_i+1}{d}\not \in \Z \hbox{ and } <\frac{n}{2}\right\}.
\end{equation}
Therefore,  the lattice of Hodge cycles is given by
\begin{equation}
\label{15o2017}
\hodge_n(X,\Z):=
\left \{
\delta\in H_n(X,\Z)  \Bigg| \mathlarger{\int}_{\delta}\omega_\beta=0,\ \ \forall\beta\in I_1
\right\}.
\end{equation}
Let $I_2$ be the set of $\beta\in I$ such that 
 \begin{enumerate}
  \item 
 There is an  integer $\beta_0$ satisfying the conditions $0\leq \beta_0\leq d-2$ and $\sum_{i=0}^{n+1}\frac{\beta_i+1}{d}:=\frac{n}{2}+1$.
 \item 
 There is no permutation $\sigma$ of $0,1,\ldots,n+1$ without fixed point and with $\sigma^2$ being the
identity such that $\beta_i+\beta_{\sigma(i)}=d-2$ for all $i$.
 \end{enumerate}
 The space of linear Hodge cycles is defined in the following way:
 \begin{equation}
\label{27j2016}
V^d_n:=
\left \{
\delta\in H_n(X,\Z)  \Bigg| \mathlarger{\int}_{\delta}\omega_\beta=0,\ \ \forall\beta\in I_1\cup I_2
\right\}.
\end{equation}
We know that the topological classes of the primitive linear algebraic cycles $\check{P}_i$ are in $(V^d_n)_\prim$ (the subindex $0$ was introduced at the beginning of $\S 3$) 
and over rational numbers they generate it, see for instance \cite[Theorem A]{Aoki1983}, (the first statement also follows from the integration
formula in the main theorem of \cite{Roberto}). Moreover, if $d$ is prime or $d=4$ or $d$ is coprime with $(n+1)!$ then $I_2$ is the empty
set and so $\hodge_n(X,\Z)=V^d_n$. 
One can easily compute the integrals  $\int_{\delta_\beta}\omega_{\beta'},\ \beta\in I_1\cup I_2$ explicitly, see \cite[Lemma 7.12]{DMOS} or  \cite[Proposition 15.1]{ho13}, 
and it turns out that 
$\delta:=\sum_{\beta\in I} {c_\beta}\delta_\beta\in (V^d_n)_\prim,\ \ c_\beta\in \Z$  if and only if the coefficient matrix 
${\tt C}=[c_\beta]_{1\times \mu}$ satisfies 
${\tt C}\cdot Q =0$, where $Q$ is  the $\mu\times \check\mu$ matrix given by:
\begin{equation}
Q:= \left[ \prod_{i=1}^{n+1}\left( 
  \zeta_d^{(\beta_i+1)(\beta_i'+1)}-\zeta_d^{\beta_i(\beta_i'+1)}
  \right)\right]_{\beta\in I ,\beta'\in I_1\cup I_2}
 \end{equation}
$\zeta_d:=e^{\frac{2\pi i}{d}}$ is the $d$-th primitive root of unity and $\check\mu:=\# (I_1\cup I_2)$.  
Since the minimal polynomial of $\zeta_d$
is of degree $\varphi(d)$ (Euler's totient function) we write $Q=\sum_{i=0}^{\varphi(d)-1} Q_i \zeta_d^i$ 
and define 
\begin{equation}
 \label{20.10.2017}
 {\tt A}_2:=[Q_0|Q_1|\cdots| Q_{\varphi(d)-1}]
\end{equation}
which is a $\mu\times (\check\mu\cdot  \varphi(d))$ matrix obtained by concatenation of $Q_i$'s.
Now, the coefficients matrix ${\tt C}$ of a  primitive Hodge cycle satisfies $\tt C\cdot {\tt A}_2=0$. 
We write ${\tt A}_2$ in the Smith normal form. 
Let $\tt Y$ be the $(\mu-{ m})\times \mu$ matrix which is obtained
by joining the $(\mu-{m})\times {m}$ zero matrix with the $(\mu- m)\times (\mu-m)$ identity matrix in a row.
The rows of 
\begin{equation}
\label{23june2016}
{\tt X}:=\tt Y\cdot   {\tt U}_2
\end{equation}
form a basis of the integer-valued $1\times \mu $ matrices $\tt C$ such that  $\tt C\cdot {\tt A}_2=0$. The intersection numbers of topological cycles $\delta_\beta$
are given by the  the following rules:
\begin{eqnarray}
\nonumber
\langle \delta_{\beta},\delta_{\beta'}\rangle&=&(-1)^{n}\langle \delta_{\beta'},\delta_{\beta}\rangle,\ \ \ 
 \forall \beta,\beta'\in I,\\ \label{27.1.2017}
\langle \delta_{\beta},\delta_{\beta}\rangle &=& (-1)^{\frac{n(n-1)}{2}}(1+(-1)^{n}),\ \ \ \forall\beta\in I \\ \nonumber
\langle \delta_{\beta},\delta_{\beta'}\rangle &=& 
(-1)^{\frac{n(n+1)}{2}}
(-1)^{\Sigma_{k=1}^{n+1} \beta_k'-\beta_k}
\end{eqnarray}
for those $\beta,\beta'\in I$  such that for all   $k=1,2,\ldots,n+1$ we have 
$\beta_k\leq \beta_k'\leq \beta_k+1$ and $\beta\neq \beta'$.
In the remaining cases, except those arising from the previous ones by a permutation, we have
$\langle \delta_{\beta},\delta_{\beta'}\rangle=0$, see \cite[Proposition 7.7]{ho13}. 
This computation of intersection number of $\delta_\beta$'s is essentially due to  F. Pham,  see \cite{Pham1965} and \cite[Page 66]{arn}. 
Let $\Psi:=[\langle \delta_{\beta},\delta_{\beta'}\rangle]$
be the $\mu\times\mu$ intersection matrix.  The intersection matrix for primitive Hodge cycles is:
 $$
 {\tt A}_3:={\tt X}\cdot \Psi \cdot {\tt X}^{\tr}.
 $$ 
 %\subsection{The lattice of primitive Hodge cycles supported in linear algebraic cycles:} 

\subsection{Proof of Theorem \ref{11oct2017}}
The proof is reduced to computation of elementary divisors of ${\tt A}_1$ and ${\tt A}_3$ (counting with repetitions) and observing that they are 
the same. In other words, one has to check if the output of the algorithms (\ref{LinCycAlg}) and (\ref{PrimCycAlg}) are the same. Such elementary divisors might be obtained from those listed in Table \ref{elementary-divs} in the following way: decrease  the power of $1$ by two  and add   $d$ to the list of elementary divisors.
This proves that the sublattice of $H_n(X,\Z)_\prim$ generated by \eqref{15sept2017} coincides with $(V^d_n)_\prim$ (the lattice of primitive 
Hodge cycles).
Since $H_n(X,\Z)=H_n(X,\Z)_\prim+\Z \cdot P_1$ we get Theorem \ref{11oct2017}. We had to compute 
elementary divisors of $V_n^d$ in  Table \ref{elementary-divs} separately, as we did not find any straightforward argument 
relating this  with the elementary divisors of its primitive part. For the sake of completeness, we also computed the sign of each discriminant  in Table \ref{elementary-divs}. This is   $\det(\tt U_1)\cdot\det(\tt T_1)$ for the Smith normal form of ${\tt A}_1$ (one could also use ${\tt A}_3$ for this purpose). The direct computation of the determinant in general does not work, and we had to compute the Smith normal form of ${\tt U}_1, {\tt T}_1$ and count the number of $-1$ in the corresponding two matrices $\tt S$.    

\begin{algorithm}[ht!]
\label{LinCycAlg}
\KwData{Pair $(n,d)$, where $n$ is an even positive integer and $d$ a natural number, representing the dimension and the degree, respectively, of a Fermat variety $X^d_n \subseteq \licy{n+1}$ .}
\KwResult{Elementary divisors of the intersection matrix of linear algebraic cycles of the Fermat variety $X^d_n \subseteq \licy{n+1}$.}
	\Begin{
		$\mathcal{I} \leftarrow $ Family of ideals of the subvarieties $\licy{\frac{n}{2}}_{a,b}$ described in (\ref{ref2})\;
		\For{ $(\mathfrak{a}, \mathfrak{b}) \in \mathcal{I} \times \mathcal{I}$ }{
			$\mathfrak{c} \leftarrow$ Ideal generated by $\mathfrak{a}$ and $\mathfrak{b}$\;
			$m \leftarrow \dim\left ( \mathfrak{c} \right )$\;
			$M_{\mathfrak{a}, \mathfrak{b}} \leftarrow  \dfrac{1 - (-d + 1)^{m+1}}{d}$\; 
		}
		$L \leftarrow $ List of elementary divisors obtained from the Smith decomposition of $M$\;
		\Return $L$\;

	}
\caption{Computation of elementary divisors of the intersection matrix of linear algebraic cycles for a Fermat variety $X^d_n \subseteq \licy{n+1}$.}
\end{algorithm}

\begin{algorithm}[H]
\label{PrimCycAlg}
\KwData{Pair $(n,d)$, where $n$ is an even positive integer and $d$ a natural number, representing the dimension and the degree, respectively, of a Fermat variety $X^d_n \subseteq \licy{n+1}$ .}
\KwResult{Elementary divisors of the intersection matrix of primitive Hodge cycles of the Fermat variety $X^d_n \subseteq \licy{n+1}$.}
	\Begin{
		$I, I_1, I_2 \leftarrow$ index sets as defined in subsection \ref{primitive_Hodge_cycles}\;
		$Q \leftarrow \left[ \prod_{i=1}^{n+1}\left( \zeta_d^{(\beta_i+1)(\beta_i'+1)}-\zeta_d^{\beta_i(\beta_i'+1)} \right)\right]_{\beta\in I ,\beta'\in I_1\cup I_2}$\;
		${\tt A_2} \leftarrow [Q_0|Q_1|\cdots| Q_{\varphi(d)-1}] $ where $Q_i$ satisfy $Q = \sum_{i= 0}^{\varphi(d)-1}Q_i \zeta_d^i$\;
		${\tt U_2} \leftarrow $ matrix arising in Smith normal form of ${\tt A_2}$ as in (\ref{13oct2017})\;
		${\tt X} \leftarrow $ defined as in (\ref{23june2016})\;
		$\Psi \leftarrow \left [ \langle \delta_{\beta},\delta_{\beta'}\rangle \right]$, see (\ref{27.1.2017})\;
		${\tt A_3} \leftarrow {\tt X}\cdot \Psi \cdot {\tt X}^{\tr}$\;
		$L \leftarrow $ List of elementary divisors obtained from the Smith decomposition of ${\tt A_3}$\;
		\Return $L$\;

	}
\caption{Computation of elementary divisors of the intersection matrix of primitive Hodge cycles for a Fermat variety $X^d_n \subseteq \licy{n+1}$.}
\end{algorithm}
 \section{Our computational experience}
\label{16oct2017swap}
The matrix ${\tt A_1}$ in \eqref{11.10.2017} is implemented in {\tt Matlab}, ${\tt A}_2$ in {\tt Singular},  and ${\tt A_3}$ in {\tt Mathematica}. 
All the Smith normal forms are performed in {\tt Mathematica}. 
For all data produced in our computations see the \href{http://w3.impa.br/~aljovin/hodge_project.html}{first author's web page.}
We started to prepare Table \ref{elementary-divs} with a computer with processor {\tt  Intel  Core i7-7700}, $16$ GB Memory plus $16$ GB swap memory and the operating system {\tt Ubuntu 16.04}. 
It turned out that for the case $(n,d)=(4,6)$ we get the  `Memory Full' error. This was during the computation of the Smith normal form of the matrix ${\tt A}_1$. Therefore,  we had to increase the swap memory up to $170$ GB. Despite the low speed of the swap which slowed down the computation, the computer was able to use the data and give us the desired output. For this reason the computation in this case took more than two days. We only know that at least $56$ GB of the swap were used. We were
able to compute the elementary divisors of the lattice of primitive Hodge cycles in Table \ref{ed}  but were not able to compute the elementary divisors of the lattice of linear algebraic cycles.
\begin{table}[ht!]
\centering
\begin{tabular}{ll}
\hline
\multicolumn{1}{ | l | }{$(n, d)$} & \multicolumn{1}{l | }{Elementary divisors}                                    \\ \hline 
\multicolumn{1}{ | l | }{$(6,3)$}  & \multicolumn{1}{l | }{$+1^{54} \cdot 3^{8} \cdot 9^{7} \cdot 27^{1}$}                                                                                                                                                                              \\ \hline 
\multicolumn{1}{ | l | }{$(6,4)$}  & \multicolumn{1}{l | }{$-1^{576} \cdot 2^{68} \cdot 4^{49} \cdot 8^{296} \cdot 16^{48} \cdot 32^{69} \cdot 512^{2}$}                                                                                                                             \\ \hline 
\multicolumn{1}{ | l | }{$(8,3)$}  & \multicolumn{1}{l | }{$-1^{172} \cdot 3^{35} \cdot 9^{34} \cdot 27^{11}$}                                                                                                                                                                          \\ \hline 
\multicolumn{1}{ | l | }{$(10,3)$} & \multicolumn{1}{l | }{$+1^{559} \cdot 3^{144} \cdot 9^{144} \cdot 27^{76} \cdot 81^{1} $}                                                                                                                                                          \\ \hline 
\end{tabular}
\caption{Elementary divisors of the lattice of primitive Hodge cycles of $X^d_n$}
\label{ed}
\end{table}

For the case $(n,d)=(10,3)$, $36$ GB of swap were used and the number of linear algebraic cycles is $N=7577955$ and we were even not able to produce the 
 matrix ${\tt A_1}$ which is the intersection 
 matrix of the lattice of linear algebraic cycles.

 \def\cprime{$'$} \def\cprime{$'$} \def\cprime{$'$} \def\cprime{$'$}

%\bibliography{biblio}
%\bibliographystyle{alpha}

\end{document}